\documentclass[11pt]{amsart}
\usepackage{amsfonts}
\usepackage{amssymb,latexsym}
\usepackage{amsmath}
\usepackage{amsthm}
\usepackage{amssymb}
\usepackage{enumerate}
\newtheorem{theorem}{Theorem}[section]
\newtheorem{lemma}[theorem]{Lemma}

\theoremstyle{definition}

\newtheorem{remark}[theorem]{Remark}
\numberwithin{equation}{section}

\def\dsum{\displaystyle\sum}

\newcommand{\al}{\alpha}
\newcommand{\la}{\lambda}
\begin{document}
\title[ Generalization of Newton-Maclaurin's
inequalities]{A generalization of Newton-Maclaurin's inequalities}

\author{Changyu Ren}
\address{Changyu Ren, School of Mathematical Science, Jilin University, Changchun, 130012, Jilin Province, P.R. China}\email{rency@jlu.edu.cn}

\thanks{Research of author is supported  by an NSFC Grant No. 11871243}

\begin{abstract}
In this paper, we prove Newton-Maclaurin type inequalities for
functions obtained by linear combination of two neighboring primary
symmetry functions, which is a generalization of the classical
Newton-Maclaurin inequality.

\end{abstract}

\keywords {Primary symmetry function, Newton-Maclaurin's inequality,
Hessian equations.}

\subjclass{26D05, 26D15}

\maketitle

\section{Introduction}

The $k$-th elementary symmetric function of the variables
$x_1,x_2,\cdots,x_n$ is defined by
$$
\sigma_k(x)=\dsum_{1\leq i_1<\cdots<i_k\leq n}x_{i_1}\cdots x_{i_k},
\quad 1\leq k\leq n,
$$
where $x=(x_1,x_2,\cdots,x_n)$. For example, when $n = 3$,
\begin{align*}
\sigma_1(x)=x_1+x_2+x_3,\quad \sigma_2(x)=x_1x_2+x_1x_3+x_2x_3,\quad
\sigma_n(x)=x_1x_2x_3.
\end{align*}
It will be convenient to define $\sigma_0(x)=1$, and define
$\sigma_k(x)=0$ if $k<0$ or $k>n$. Furthermore, define a $k$-th
elementary symmetric mean as
$$
E_k(x)= \dfrac{\sigma_k(x)}{C_n^k},\quad k=0,1,\cdots,n,
$$
where $C_n^k=\dfrac{n!}{k!(n-k)!}$.

\begin{theorem}\label{th1}
(Newton \cite{Newt} and Maclaurin \cite{Mac}) Let
$x=(x_1,\cdots,x_n)$ be an $n$-tuple non-negative real numbers. Then
\begin{equation}\label{e1.1}
E_k^2(x)\geq E_{k-1}(x)E_{k+1}(x),\quad k=1,2\cdots,n-1,
\end{equation}
\begin{equation}
E_1(x)\geq E_2^{1/2}(x)\geq\cdots\geq E_n^{1/n}(x),
\end{equation}
and the inequality is strict unless all entries of $x$ coincide.
\end{theorem}
\par
 The inequalities (\ref{e1.1}) in Theorem \ref{th1} is a consequence of a rule stated by Newton \cite{Newt}
 which gives a lower bound on the number of nonreal roots of a real
 polynomial. Since Newton did not give a proof of his rule, the
 proof of Theorem \ref{th1} is due to MacLaurin \cite{Mac}.
 Actually, the Newton's inequalities (\ref{e1.1}) work for $n$-tuples
 of real, not necessarily positive elements. For an inductive proof in the case where
 $x_1,x_2,\cdots, x_n$ are nonnegative, see \cite{HLP} $\S$ 2.22. For a proof by differential calculus in the case
 where $x_1,x_2,\cdots, x_n$ are real, see \cite{HLP} $\S$ 4.3, or \cite{Ros}.
 Several reformulations/generalisations of Newton's inequalities have been given over the years,
 see for instance \cite{Men,Ros,Nic,EH}.

\par
The Newton-Maclaurin inequalities play important roles in deriving
theoretical result for fully nonlinear partial differential
equations and geometric analysis. There are many important results
that need to use the Newton-Maclaurin inequalities, such as
\cite{gg,GM,GMZ,GRW,GW,HS} etc. This is because the following
$k$-Hessian equations and curvature equations
\begin{align*}
\sigma_k(\la(u_{ij}))&=f(x,u,\nabla u), \\
\sigma_k(\kappa(X))&=f(X,\nu)
\end{align*}
are central studies in the field of fully nonlinear partial
differential equations and geometric analysis. Their left-hand side
$k$-Hessian operator $\sigma_k$ is the primary symmetric function
about the eigenvalues $\lambda=(\la_1,\cdots,\la_n)$ of the Hessian
matrix $(u_{ij})$ or
 the principal curvature
$\kappa=(\kappa_1,\cdots,\kappa_n)$ of the surface. In recent years,
the fully non-linear equations derived from linear combinations of
primary symmetric functions have received increasing attention. For
example, the following special Lagrangian equation
$$
{\rm Im det}(\delta_{ij}+{\rm
i}u_{ij})=\dsum_{k=0}^{[(n-1)/2]}(-1)^k\sigma_{2k+1}(\la(u_{ij}))=0,
$$
which derived by Harvey and Lawson in their study of the minimal
submanifold problem \cite{HL} . The fully non-linear partial
differential equations
$$
P_m(u_{ij})=\dsum_{k=0}^{m-1}(l_k^{+})^{m-k}(x)\sigma_k(u_{ij})=g^{m-1}(x)
$$
studied by Krylov \cite{Kry} and Dong \cite{Dong}. In \cite{LiRW},
Li-Ren-Wang studied the concavity of operators $
\dsum_{s=0}^k\al_s\sigma_s$ and $\sigma_k+\al\sigma_{k-1}$, and
discussed the curvature estimates for the convex surface of the
corresponding equations
$$
\dsum_{s=0}^k\al_s\sigma_s(\kappa(X))=f(X,\nu(X)), \quad X\in M.
$$
Guan and Zhang \cite{GZ} investigated the curvature estimates of the
following curvature equation
$$
\sigma_k(W_u(x))+\alpha(x)\sigma_{k-1}(W_u(x))=\dsum_{l=0}^{k-2}\alpha_{l}(x)\sigma_{l}(W_u(x)),\quad
x\in \mathbb{S}^n.
$$
Recently, Liu and Ren \cite{LR} discussed the Pogorelov-type $C^2$
estimates for $(k-1)$-convex and $k$-convex solutions of the
following Sum Hessian equations
$$
\sigma_{k}(\lambda(u_{ij}))+\al\sigma_{k-1}(\lambda(u_{ij}))=f(x,u,\nabla
u),
$$
and established a rigidity theorem when the right-hand side of the
equation is constant.

\par
In the further study of the above problem, the Newton-Maclaurin type
inequalities for the operators derived from the left-hand side of
the equations are always needed. A natural question is that whether
the Newton-Maclaurin type inequalities for the operators of linear
combinations of these primary symmetric functions still hold?
\par
Now the main result of this paper is stated as follows.
\begin{theorem}\label{th2}
Let $n\geq 3, 1\leq k\leq n-2$. For any real number $\al\in
\mathbb{R}$ and any $n$-tuple real numbers
$x=(x_1,x_2,\cdots,x_n)\in \mathbb{R}^n$, we have
\begin{equation}\label{e1.3}
[\al E_k(x)+E_{k+1}(x)]^2\geq [\al E_{k-1}(x)+E_k(x)][\al
E_{k+1}(x)+E_{k+2}(x)].
\end{equation}
The inequality is strict unless $x_1=\cdots=x_n$, or
$$
\dfrac{\al E_{k}(x)+E_{k+1}(x)}{\al E_{k-1}(x)+E_{k}(x)}=\dfrac{\al
E_{k+1}(x)+E_{k+2}(x)}{\al E_{k}(x)+E_{k+1}(x)}=-\al.
$$
A special case in which the above equality holds is that there are
$n-1$ elements of $x_1, x_2, \cdots, x_n$ taking the value
$-\alpha$.
\par
 Furthermore, if $\al\geq 0$ and
$$
\al E_m(x)+E_{m+1}(x)\geq 0, \quad {\rm for~all~~ } m=0,1,\cdots,k,
$$
then
\begin{equation}\label{e1.4}
[\al+E_{1}(x)]\geq [\al E_{1}(x)+E_2(x)]^{1/2}\geq\cdots\geq [\al
E_{k}(x)+E_{k+1}(x)]^{1/(k+1)}.
\end{equation}

\end{theorem}
\par
Theorem \ref{th2} is the Newton-Maclaurin type inequality for the
function obtained by linear combination of two neighboring terms
$E_k$ and $E_{k+1}$. Obviously, it is the classical Newton-Maclaurin
inequality when $\al=0$. Furthermore, one may ask if the
Newton-Maclaurin type inequality still holds for the functions which
obtained by linear combination of general primary symmetric
functions. That is, for $\al=(\al_1,\cdots,\al_n)\in\mathbb{R}^n$,
$x=(x_1,\cdots,x_{n+1})\in \mathbb{R}^{n+1}$, whether the following
inequalities hold?
\begin{equation}\label{e1.5}
\left(\dsum_{k=1}^n\al_kE_k(x)\right)^2\geq
\left(\dsum_{k=1}^n\al_kE_{k-1}(x)\right)\left(\dsum_{k=1}^n\al_kE_{k+1}(x)\right).
\end{equation}
\par
The answer is no in general. For example, let $n=3$, $\al=(1,0,1)$,
$x=(4,4,\dfrac{1}{4},\dfrac{1}{4})$, we have
\begin{align*}
E_1(x)=&\dfrac{1}{4}(4+4+\dfrac{1}{4}+\dfrac{1}{4})=\dfrac{17}{8},\\
E_2(x)=&\dfrac{1}{6}(4\times 4+4\times 4\times\dfrac{1}{4}+\dfrac{1}{4}\times\dfrac{1}{4})=\dfrac{107}{32},\\
E_3(x)=&\dfrac{1}{4}(2\times 4\times 4\times\dfrac{1}{4}+2\times 4\times\dfrac{1}{4}\times\dfrac{1}{4})=\dfrac{17}{8},\\
E_4(x)=&4\times 4 \times\dfrac{1}{4}\times\dfrac{1}{4}=1,
\end{align*}
therefore,
\begin{align*}
[E_1(x)+E_3(x)]^2-[1+E_2(x)][E_2(x)+E_4(x)]=-\dfrac{825}{1024}<0.
\end{align*}
Maybe inequalities \eqref{e1.5} hold when $\al=(\al_1,\cdots,\al_n)$
satisfy some structural conditions.

\par
A straightforward conclusion of Newton's inequality \eqref{e1.1} is
\begin{equation}\label{e1.6}
\sigma_k^2(x)-\sigma_{k-1}(x)\sigma_{k+1}(x)\geq\theta
\sigma_k^2(x),
\end{equation}
where $0<\theta<1$ is a constant depending on $n$ and $k$. The
inequality \eqref{e1.6} is also a common sense which is widely used
in the study of $k$-Hessian equations and curvature equations, see
\cite{GM,GMZ,GRW,GW,HS,LiRW,LRW,LT}, etc. For the Sum Hessian
operators $\sigma_{k}+\al\sigma_{k-1}$, if $\al>0$ and
$x=(x_1,\cdots,x_n)\in\Gamma_k$, Liu and Ren \cite{LR} proved
$$
[\sigma_{k}(x)+\al\sigma_{k-1}(x)]^2\geq[\sigma_{k-1}(x)+\al\sigma_{k-2}(x)][\sigma_{k+1}(x)+\al\sigma_{k}(x)]
$$
by Newton's inequality \eqref{e1.6}. Here $\Gamma_k$ is the
Garding's cone
\[\Gamma_k=\{\lambda \in \mathbb R^n \ | \quad \sigma_m(\lambda)>0, \quad  m=1,\cdots,k\}.\]

\par
The requirement $x\in\Gamma_k$ of above inequalities restrict their
applications in the study of Sum Hessian equations. In this paper,
we remove the restrictions from $\al$ and $x$.

\begin{theorem}\label{th3}
Let $n\geq 3, 0\leq k\leq n-1$. For any real number $\al\in
\mathbb{R}$ and any $n$-tuple real numbers
$x=(x_1,x_2,\cdots,x_n)\in \mathbb{R}^n$, we have
\begin{equation}\label{e1.7}
(1-\theta)[\al\sigma_{k}(x)+\sigma_{k+1}(x)]^2-[\al\sigma_{k-1}(x)+\sigma_{k}(x)][\al\sigma_{k+1}(x)+\sigma_{k+2}(x)]\geq
0,
\end{equation}
where $0<\theta<1$ is a constant depending only on $n$ and $k$.
\end{theorem}

\begin{remark}
When $k=0$ or $k=n-1$, \eqref{e1.3} does not hold, one can take
$x_1=x_2=\cdots=x_n$ for verification. It is a difference between
\eqref{e1.7} and \eqref{e1.3}.
\end{remark}

\par
\section{The proof of Theorem \ref{th2}}

\par
First, we consider the following inequality in Euclidean space
$\mathbb{R}^3$.
\begin{lemma}\label{y2.1}
For any real number $\al\in \mathbb{R}$ and $z=(z_1,z_2,z_3)\in
\mathbb{R}^3$, we have
\begin{equation}
[\al E_1(z)+E_{2}(z)]^2\geq [\al E_0(z) +E_1(z)][\al
E_{2}(z)+E_{3}(z)].
\end{equation}
The inequality is strict unless $z_1=z_2=z_3$ or
$$
\dfrac{\al E_1(z)+E_{2}(z)}{\al E_0(z) +E_1(z)}=\dfrac{\al
E_2(z)+E_{3}(z)}{\al E_1(z) +E_2(z)}=-\al.
$$
\end{lemma}
\begin{proof}
A straightforward calculation shows
\begin{align*}
&18[\al E_1(z)+E_{2}(z)]^2-18[\al E_0(z) +E_1(z)][\al E_{2}(z)+E_{3}(z)]\\
=&[(z_1+\al)(z_2+\al)-(z_1+\al)(z_3 + \al)]^2+[(z_1 + \al)(z_2 +
\al) - (z_2
+\al)(z_3 + \al)]^2\\
& + [(z_1 + \al)(z_3 + \al) - (z_2 + \al)(z_3 +\al)]^2\geq 0.
\end{align*}
Obviously, the inequality is strict unless $z_1=z_2=z_3$ or any two
elements of $z_1,z_2,z_3$ valued $-\al$. Without loss of generality,
we assume $z_1=z_2=-\al$. Then we obtain
\begin{align*}
\al E_0(z) +E_1(z)=\dfrac{z_3+\al}{3}, \al
E_1(z)+E_{2}(z)=\dfrac{-\al(z_3+\al)}{3},  \al
E_2(z)+E_{3}(z)=\dfrac{\al^2(z_3+\al)}{3}.
\end{align*}
Thus,
$$
\dfrac{\al E_1(z)+E_{2}(z)}{\al E_0(z) +E_1(z)}=\dfrac{\al
E_2(z)+E_{3}(z)}{\al E_1(z) +E_2(z)}=-\al.
$$
\end{proof}

\par
The following Lemma is a useful tool to prove Newton's inequalities
\eqref{e1.1} from  Sylvester \cite{Syl,Nic}.

\begin{lemma}\label{y2.2}
If
\begin{equation*}
F(x,y)=c_0x^m+c_1x^{m-1}y+\cdots+c_m y^m
\end{equation*}
is a homogeneous function of the $n$-th degree in $x$ and $y$ which
has all its roots $x/y$ real, then the same is true for all
non-identical 0 equations $\dfrac{\partial^{i+j}F}{\partial
x^j\partial y_j}=0$, obtained from it by partial differentiation
with respect to $x$ and $y$. Further, if $E$ is one of these
equations, and it has a multiple root $\al$, then $\al$ is also a
root, of multiplicity one higher, of the equation from which $E$ is
derived by differentiation.
\end{lemma}
Let us now present the proof of Theorem \ref{th2}.
\par
\begin{proof}
\par
We assume $P(t)$ is a polynomial of $n$-degree, with real roots
$x_1,x_2,\cdots x_n$. Then $P(t)$ is represented as
$$
P(t)=\prod_{i=1}^n(t-x_i)=E_0(x)t^n-C_n^1E_1(x)t^{n-1}+C_n^2E_2(x)t^{n-2}-\cdots+(-1)^nE_n(x),
$$
and we shall apply Lemma \ref{y2.2} to the associated homogeneous
polynomial
$$
F(t,s)=E_0(x)t^n-C_n^1E_1(x)t^{n-1}s+C_n^2E_2(x)t^{n-2}s^2-\cdots+(-1)^nE_n(x)s^n.
$$
Considering the case of the
derivatives$\dfrac{\partial^{n-3}F}{\partial t^{n-2-k}\partial
s^{k-1}}$(for $k=1,\cdots,n-2$), we arrive to the fact that all the
cubic polynomials
\begin{equation}\label{e2.2}
E_{k-1}(x)t^3-3E_k(x)t^2s+3E_{k+1}(x)ts^2-E_{k+2}(x)s^3
\end{equation}
 for
$k=1,\cdots,n-2$ also have real roots.
\par
If $E_{k-1}(x)=E_{k+2}(x)=0$, it is easy to get
\begin{align*}
&[\al E_k(x)+E_{k+1}(x)]^2- [\al E_{k-1}(x)+E_k(x)][\al
E_{k+1}(x)+E_{k+2}(x)]\\
=&\dfrac{1}{2}[\al
E_k(x)+E_{k+1}(x)]^2+\dfrac{\al^2}{2}E_k^2(x)+\dfrac{1}{2}E_{k+1}^2(x)\geq
0.
\end{align*}
So we divide it in two cases to deal with (\ref{e1.3}):
$E_{k-1}(x)\neq 0$ or $E_{k+2}(x)\neq 0$. \vskip 3mm
\par
Case A:  When $E_{k-1}(x)\neq 0$, by \eqref{e2.2}, the polynomial
$$t^3-\dfrac{3E_k(x)}{E_{k-1}(x)}t^2+\dfrac{3E_{k+1}(x)}{E_{k-1}(x)}t-\dfrac{E_{k+2}(x)}{E_{k-1}(x)}$$
has three real roots. We denote the roots by $z_1,z_2,z_3$ and
denote $z=(z_1,z_2,z_3)$, then
$$
E_1(z)=\dfrac{E_{k}(x)}{E_{k-1}(x)},\quad
E_2(z)=\dfrac{E_{k+1}(x)}{E_{k-1}(x)},\quad
E_3(z)=\dfrac{E_{k+2}(x)}{E_{k-1}(x)}. $$ Using Lemma \ref{y2.1} we
obtain (\ref{e1.3}). \vskip 3mm
\par
Case B: When $E_{k+2}(x)\neq 0$, by \eqref{e2.2}, the polynomial
$$s^3-\dfrac{3E_{k+1}(x)}{E_{k+2}(x)}s^2+\dfrac{3E_{k}(x)}{E_{k+2}(x)}s-\dfrac{E_{k-1}(x)}{E_{k+2}(x)}$$
has three real roots. Then the proof of (\ref{e1.3}) is similar to
Case A.
\par
By Lemma \ref{y2.1}, the inequalities of \eqref{e1.3} are strict
unless $x_1=x_2=\cdots=x_n$ or
$$
\dfrac{\al E_{k}(x)+E_{k+1}(x)}{\al E_{k-1}(x)+E_{k}(x)}=\dfrac{\al
E_{k+1}(x)+E_{k+2}(x)}{\al E_{k}(x)+E_{k+1}(x)}=-\al.
$$
A special case in which the above equality holds is that there are
$n-1$ elements of $x_1, x_2, \cdots, x_n$ valued $-\alpha$.

\par
Now, we prove (\ref{e1.4}). When $k=0$, we get from \eqref{e1.1},
\begin{align*}
[\al+E_{1}(x)]^2-[\al E_{1}(x)+E_2(x)]=&\al^2+\al
E_{1}(x)+E_{1}^2(x)-E_{2}(x)\\
\geq&\al(\al+E_{1}(x))\geq0.
\end{align*}
Therefore,
$$[\al+E_{1}(x)]\geq [\al E_{1}(x)+E_2(x)]^{1/2}.$$
\par
For $1\leq k\leq n-2$, we assume
$$[\al
E_{k-2}(x)+E_{k-1}(x)]^{1/(k-1)}\geq [\al
E_{k-1}(x)+E_{k}(x)]^{1/k}.$$ Combining (\ref{e1.3}) and the
inequalities above, we get
\begin{align*}
[\al E_{k-1}(x)+E_{k}(x)]^2\geq &[\al E_{k-2}(x)+E_{k-1}(x)][\al
E_{k}(x)+E_{k+1}(x)]\\
\geq &[\al E_{k-1}(x)+E_{k}(x)]^{(k-1)/k}[\al E_{k}(x)+E_{k+1}(x)].
\end{align*}
Thus,
$$
[\al E_{k-1}(x)+E_{k}(x)]^{1/k}\geq [\al
E_{k}(x)+E_{k+1}(x)]^{1/(k+1)}
$$
by canceling the common factor.

\end{proof}

\section{The proof of Theorem \ref{th3}}

\par
We divide it in two cases to prove Theorem \ref{th3}:
\par
(i) Special cases: $k=0$, $k=n-1$ or $n=3, k=1$;
\par
(ii) General cases: $n\geq 4, 1\leq k\leq n-2$.

\subsection{Special cases}
~
\par
$\bullet$ When $k=0$, choose $\theta=\dfrac{1}{2}$ in \eqref{e1.7}.
Using the identity $\sigma_1^2(x)=\dsum_{i=1}^nx_i^2+2\sigma_2$, we
have
\begin{align*}
\dfrac{1}{2}[\al+\sigma_{1}(x)]^2-[\al\sigma_{1}(x)+\sigma_{2}(x)]
=&\dfrac{\al^2}{2}+\dfrac{1}{2}\sigma_1^2(x)-\sigma_2(x)\geq 0.
\end{align*}
\par
$\bullet$ When $k=n-1$, choose $\theta=\dfrac{1}{2}$. By using the
identity
$$
\sigma_{n-1}^2(x)=\dsum_{i=1}^n\dfrac{\sigma_n^2(x)}{x_i^2}+2\sigma_{n-2}(x)\sigma_n(x),
$$
we get
\begin{align*}
&\dfrac{1}{2}[\al\sigma_{n-1}(x)+\sigma_{n}(x)]^2-[\al\sigma_{n-2}(x)+\sigma_{n-1}(x)]\al\sigma_n(x)\\
=&\dfrac{\al^2}{2}\sigma_{n-1}^2(x)+\dfrac{1}{2}\sigma_{n}^2(x)-\al^2\sigma_{n-2}(x)\sigma_n(x)
\geq 0.
\end{align*}
\par
$\bullet$ When $n=3, k=1$, choose $\theta=\dfrac{1}{2}$. Using the
identities
$$
\sigma_1(x)^2=x_1^2+x_2^2+x_3^2+2\sigma_2(x),\quad
\sigma_2^2(x)=x_1^2x_2^2+x_1^2x_3^2+x_2^2x_3^2+2\sigma_1(x)\sigma_3(x),
$$
we obtain
\begin{align*}
&\dfrac{1}{2}[\al\sigma_{1}(x)+\sigma_{2}(x)]^2-[\al+\sigma_{1}(x)][\al\sigma_2(x)+\sigma_3(x)]\\
=&\dfrac{\al^2}{2}\sigma_{1}^2(x)+\dfrac{1}{2}\sigma_{2}^2(x)-\al^2\sigma_2(x)-\al\sigma_3(x)-\sigma_{1}(x)\sigma_3(x)\\
\geq&\dfrac{\al^2}{2}x_1^2+\dfrac{1}{2}x_2^2x_3^2-\al x_1x_2x_3 \geq
0.
\end{align*}

\subsection{General cases}

For $n\geq 4, 1\leq k\leq n-2$, we denote
\begin{equation}\label{e3.1}
a=C_n^{k-1},\quad b=C_n^{k}, \quad c=C_n^{k+1}, \quad d=C_n^{k+2}.
\end{equation}
Then the inequalities \eqref{e1.7} can be rewritten as
\begin{equation*}
(1-\theta)[\al b E_{k}(x)+cE_{k+1}(x)]^2-[\al
aE_{k-1}(x)+bE_{k}(x)][\al cE_{k+1}(x)+dE_{k+2}(x)]\geq 0.
\end{equation*}
From Lemma \ref{y2.2}, and by a similar proof to \eqref{e1.3}, we
shall complete the proof  Theorem 1.3 via establishing the following
inequalities.
\begin{equation*}
(1-\theta)[\al b E_{1}(z)+cE_{2}(z)]^2-[\al a+bE_{1}(z)][\al
cE_{2}(z)+dE_{3}(z)]\geq 0.
\end{equation*}
Here $z=(z_1,z_2,z_3)\in \mathbb{R}^3$. Furthermore, it is
equivalent to prove
\begin{equation}\label{e3.2}
L(z):=(\al b \sigma_1+c\sigma_2)^2-(3\al a+b\sigma_1)(\al
c\sigma_2+3d\sigma_3)\geq \theta(\al b \sigma_1+c\sigma_2)^2,
\end{equation}
where
$$\sigma_1=z_1+z_2+z_3, \quad \sigma_2=z_1z_2+z_1z_3+z_2z_3,\quad
\sigma_3=z_1z_2z_3.$$

\par

Now we prove \eqref{e3.2}. Let us expand the left-hand and
right-hand sides of \eqref{e3.2} separately. By using identities
\begin{align*}
\sigma_1^2=z_1^2+z_2^2+z_3^2+2\sigma_2,\quad
\sigma_2^2=z_1^2z_2^2+z_1^2z_3^2+z_2^2z_3^2+2\sigma_1\sigma_3,
\end{align*}
we have
\begin{align}\label{e3.3}
L(z)=&\al^2 b^2\sigma_1^2+\al bc\sigma_1\sigma_2+c^2\sigma_2^2-3\al^2 ac\sigma_2-9\al ad\sigma_3-3bd\sigma_1\sigma_3\nonumber\\
=&\al^2 b^2\dsum_{i=1}^3z_i^2+\al^2 (2b^2-3ac)\sigma_2+\al bc\sigma_1\sigma_2\nonumber\\
&+c^2\dsum_{p<q}^3z_p^2z_q^2+(2c^2-3bd)\sigma_1\sigma_3-9\al
ad\sigma_3, \\
(\al b \sigma_1+c\sigma_2)^2=&\al^2 b^2\dsum_{i=1}^3z_i^2+2\al^2
b^2\sigma_2+2\al
bc\sigma_1\sigma_2+c^2\dsum_{p<q}^3z_p^2z_q^2+2c^2\sigma_1\sigma_3.\label{e3.4}
\end{align}

\par
For any $z=(z_1,z_2,z_3)\in \mathbb{R}^3$, we define a nonnegative
function
\begin{align*}
W(z,t) =& (z_1 - z_2)^2 (\al +t z_3)^2 + (z_1 - z_3)^2 (\al +tz_2)^2
+ (z_2 - z_3)^2(\al +tz_1)^2.
\end{align*}
Expanding $W(z,t)$ we get
\begin{align}\label{e3.5}
W(z,t)=&2\al^2 \dsum_{i=1}^3z_i^2-2\al^2 \sigma_2+2\al
t\sigma_1\sigma_2+2t^2\dsum_{p<q}^3z_p^2z_q^2
-2t^2\sigma_1\sigma_3-18\al t\sigma_3.
\end{align}
We want to absorb the terms $\sigma_1\sigma_2$, $\sigma_1\sigma_3$
and $\sigma_2$ in $L(z)$ by using $(\al b\sigma_1+c\sigma_2)^2$ and
$W(z,t)$. So let
\begin{align}\label{e3.6}
L(z) =& \theta_1 (\al b \sigma_1+c\sigma_2)^2+ \theta_2 W(z,t)+V(z),
\end{align}
where $\theta_1$ and $\theta_2$ are constants which will be
determined later, $V(z)$ is the remaining part. By
\eqref{e3.3}-\eqref{e3.6}, comparing the coefficients of
$\sigma_1\sigma_2$,$\sigma_1\sigma_3$ and $\sigma_2$, we have
$$
\left\{\begin{array}{ll} bc=2 bc\theta_1 +2t\theta_2,&\quad (\sigma_1\sigma_2)\\
2c^2-3bd=2c^2\theta_1 -2t^2\theta_2,& \quad (\sigma_1\sigma_3)\\
 (2b^2-3ac)=2 b^2\theta_1 -2 \theta_2. &\quad (\sigma_2)
\end{array} \right.
$$
We solve the equations above to find
\begin{align}\label{e3.7}
\theta_1 =\dfrac{3 ( 2 a c^3 + 2 b^3 d -b^2 c^2 - 3 a b c d)}{6 a
c^3 +6 b^3 d-4 b^2 c^2 }, ~ \theta_2= \dfrac{c^2 (3ac-b^2)^2}{6 a
c^3 + 6 b^3 d-4 b^2 c^2 }, ~ t=\dfrac{b (3bd-c^2)}{c (3ac-b^2)}.
\end{align}
Using \eqref{e3.6} again, we obtain
\begin{align*}
V(z)=&A_1\al^2 \dsum_{i=1}^3z_i^2+A_2\dsum_{p<q}^3z_p^2z_q^2+A_3\al
\sigma_3,
\end{align*}
where
\begin{align*}
A_1=&b^2-b^2\theta_1 -2\theta_2, \quad
A_2=c^2-c^2\theta_1-2t^2\theta_2, \quad A_3=18t\theta_2-9ad.
\end{align*}

{\bf We claim that}  For $n\geq 4, 1\leq k\leq n-2$,
\begin{align}\label{e3.8}
\theta_1>0,\quad\theta_2> 0, \quad V(z)\geq 0.
\end{align}
 Notice that $W(z,t)$ is
nonnegative, so by \eqref{e3.6} we obtain \eqref{e3.2}. That is,
Theorem \ref{th3} is proved. Next, we will prove the claim.

\begin{lemma}\label{yl1}
For $n\geq 4$, $1\leq k\leq n-2$, we have
\par
(i) $3bd - c^2> 0$,\quad $3ac - b^2> 0$;
\par
(ii) $2 a c^3 + 2 b^3 d -b^2 c^2 - 3 a b c d>0$.
\end{lemma}
\begin{proof}
(i) By \eqref{e3.1} we have
\begin{equation}\label{e3.9}
\dfrac{a}{b}=\dfrac{C_n^{k-1}}{C_n^{k}}=\dfrac{k}{n-k+1},\quad
\dfrac{b}{c}=\dfrac{C_n^{k}}{C_n^{k+1}}=\dfrac{k+1}{n-k},\quad
\dfrac{c}{d}=\dfrac{C_n^{k+1}}{C_n^{k+2}}=\dfrac{k+2}{n-k-1}.
\end{equation}
By \eqref{e3.9}, one has
\begin{align}\label{e3.10}
3bd-c^2=&c^2(\dfrac{3bd}{c^2}-1)=c^2\left(\dfrac{3(k+1)(n-k-1)}{(n-k)(k+2)}-1\right)\nonumber\\
=&c^2\cdot\dfrac{-2k^2+2(n-2)k+(n-3)}{(n-k)(k+2)}.
\end{align}
Since the following quadratic function
\begin{align*}
f_1(k):=-2k^2+2(n-2)k+(n-3)
\end{align*}
with respect to variable $k$ is increasing on the interval
$[1,\dfrac{n-2}{2}]$, decreasing on the interval
$[\dfrac{n-2}{2},n-2]$, and
$$
f_1(1)=3(n-3)> 0,\quad  f_1(n-2)=n-3> 0,
$$
 $f_1(k)>0$ on $[1,n-2]$. Combining \eqref{e3.10}, we obtain $3bd
- c^2> 0$.
\par
 Similarly,
\begin{equation*}
3ac-b^2=b^2\cdot\dfrac{-2k^2+2nk-n-1}{(n-k+1)(k+1)},
\end{equation*}
and the function $-2k^2+2nk-n-1$ respect to variable $k$ is positive
on $[1,n-2]$,  so we get $3ac- b^2>0$.

\par
(ii) By \eqref{e3.9}, we have
\begin{align*}
2 a c^3 + 2 b^3 d -b^2 c^2 - 3 a b c
d=&b^2c^2(\dfrac{2ac}{b^2}+\dfrac{2bd}{c^2}-1-\dfrac{3ad}{bc})\\
=&b^2c^2\cdot\dfrac{2 (n+1) (-k^2+kn-k-1)}{(k+1) (k+2) (n-k+1)
(n-k)}.
\end{align*}
It is easy to get $-k^2+(n-1)k-1>0$ for $1\leq k\leq n-2$.

\end{proof}

\begin{lemma}\label{yl2}
For $n\geq 4$, $1\leq k\leq n-2$, we have
\par
(i) $A_1>0$;
\par
(ii) $A_2>0$;
\par
(iii) $A_1A_2-A_3^2/36>0$.
\end{lemma}
\par
\begin{proof}
(i) Substituting \eqref{e3.7} into $A_1$, we have
\begin{align}\label{e3.11}
A_1=&b^2(1-\theta_1) -2\theta_2=\dfrac{9ab^3cd-b^4c^2}{6 a c^3 +6
b^3 d-4 b^2 c^2 }-\dfrac{2c^2 (3ac-b^2)^2}{6 a c^3 + 6 b^3 d-4
b^2 c^2 }\nonumber\\
=&\dfrac{3c(4 a b^2 c^2 + 3 a b^3 d- 6 a^2 c^3 -b^4 c)}{6 a c^3 +6
b^3 d-4 b^2 c^2 }\nonumber\\
=&\dfrac{3ab^2c^3}{2c^2(3 a c-b^2)+2
b^2(3bd-c^2)}\cdot(4+\dfrac{3bd}{c^2}-
\dfrac{6ac}{b^2}-\dfrac{b^2}{ac}).
\end{align}
By \eqref{e3.9},
\begin{align*}
(4+\dfrac{3bd}{c^2}- \dfrac{6ac}{b^2}-\dfrac{b^2}{ac})
=&\dfrac{2(n+1)[-k^3+(n-5)k^2+(3n-2)k-n-1]}{k (k+1) (k+2 ) (n-k+1)
(n-k)}.
\end{align*}
\par
Let us show the cubic function
$$
f_2(k):=-k^3+(n-5)k^2+(3n-2)k-n-1>0.
$$
Differentiating $f_2(k)$, we solve the quadratic equation
$$
f_2'(k)=-3k^2+2(n-5)k+3n-2=0.
$$
The quadratic formula gives
$$
k_1=\dfrac{2(n-5)-\sqrt{4n^2-4n+76}}{6},\quad
k_2=\dfrac{2(n-5)+\sqrt{4n^2-4n+76}}{6}.
$$
It is easy to get
$$
k_1<\dfrac{2(n-5)-(2n-1)}{6}<0, $$
$$
\dfrac{4n-11}{6}=\dfrac{2(n-5)+(2n-1)}{6}<k_2<\dfrac{2(n-5)+3n}{6}=\dfrac{5(n-2)}{6},
$$
and $f_2'(k)\geq 0$ on $[1,k_2]$, $f_2'(k)\leq 0$ on $[k_2, n-2]$.
So function $f_2(k)$ is increasing on $[1,k_2]$ and decreasing on
$[k_2, n-2]$. On the other hand, $f_2(1)=f_2(n-2)=3(n-3)>0$, this
implied $f_2(k)>0$ on $[1, n-2]$.  Combining \eqref{e3.11} and (i)
of Lemma \ref{yl1}, we obtain $A_1>0$.

\vskip 3mm
\par
 (ii) Substituting \eqref{e3.7} into $A_2$, we have
\begin{align}\label{e3.12}
 A_2=&c^2(1-\theta_1)-2t^2\theta_2=\dfrac{9abc^3d-b^2c^4}{6 a c^3 +6
b^3 d-4 b^2 c^2 }-\dfrac{2b^2 (3bd-c^2)^2}{6 a c^3 + 6 b^3 d-4
b^2 c^2 }\nonumber\\
=&\dfrac{3 b (4 b^2 c^2 d + 3 a c^3 d - 6 b^3 d^2-b c^4)}{6 a c^3 +6
b^3 d-4 b^2 c^2 }\nonumber\\
=&\dfrac{3b^3c^2d}{2c^2(3 a c-b^2)+2
b^2(3bd-c^2)}\cdot(4+\dfrac{3ac}{b^2}-
\dfrac{6bd}{c^2}-\dfrac{c^2}{bd}).
\end{align}
By using \eqref{e3.9}, we get
\begin{align*}
(4+\dfrac{3ac}{b^2}- \dfrac{6bd}{c^2}-\dfrac{c^2}{bd})
=&\dfrac{2(n+1)[k^3-(2n+2)k^2+(n^2+3n-5)k-n^2+ 2 n-3]}{(k+1) (k+2 )
(n-k+1) (n-k)(n-k-1)}.
\end{align*}
Denote
$$
f_3(k):=k^3-(2n+2)k^2+(n^2+3n-5)k-n^2+ 2 n-3.
$$
Differentiating $f_3(k)$ and solving the quadratic equation
$$
f_3'(k)=3k^2-2(2n+2)k+n^2+3n-5=0,
$$
we get
$$
k_1=\dfrac{2(2n+2)-\sqrt{4n^2-4n+76}}{6},\quad
k_2=\dfrac{2(2n+2)+\sqrt{4n^2-4n+76}}{6}.
$$
Obviously,
$$
\dfrac{n+4}{6}=\dfrac{2(2n+2)-3n}{6}<k_1<\dfrac{2(2n+2)-(2n-1)}{6}=\dfrac{2n+5}{6},
$$
$$
k_2>\dfrac{2(2n+2)+2n-1}{6}>n,
$$
and $f_3'(k)\geq 0$ on $[1,k_1]$, $f_3'(k)\leq 0$ on $[k_1, n-2]$.
So function $f_3(k)$ is increasing on $[1,k_1]$ and decreasing on
$[k_1, n-2]$. Since $f_3(1)=f_3(n-2)=3(n-3)>0$, we have $f_3(k)>0$
on $[1, n-2]$. Combining \eqref{e3.12} and (i) of Lemma \ref{yl1},
we obtain  $A_2>0$.

\par
(iii) Substituting \eqref{e3.7} into $A_3$, we have
\begin{align*}
A_3=&18t\theta_2-9ad=\dfrac{9bc(3ac-b^2)(3bd-c^2)}{3 a c^3 +3 b^3
d-2 b^2 c^2 }-9ad\\
=&\dfrac{9 (b^3 c^3 - 3 a b c^4 - 3 b^4 c d + 11 a b^2 c^2 d - 3 a^2
c^3 d -
   3 a b^3 d^2)}{3 a c^3 +3 b^3 d-2 b^2 c^2}.
\end{align*}
By using the expressions $A_1$ and $A_2$ in \eqref{e3.11} and
\eqref{e3.12} respectively, we have
\begin{align*}
&A_1A_2-A_3^2/36\\
=&\dfrac{9 ( 12 a b^3 c^3 d+ 20 a^2 b^2 c^2 d^2-a b^2 c^5 -b^5 c^2 d
- 12 a^2 b c^4 d -
   12 a b^4 c d^2  - 3 a^3 c^3 d^2 -
   3 a^2 b^3 d^3)}{4 (3 a c^3 + 3 b^3 d-2 b^2 c^2 )}\\
=&\dfrac{9a^2b^2c^2d^2 }{4 (3 a c^3 + 3 b^3 d-2 b^2 c^2 )} \cdot
\left( \dfrac{12bc}{ad}+ 20-\dfrac{c^3}{ad^2} -\dfrac{b^3}{a^2d}
-\dfrac{12c^2}{bd}-\dfrac{12b^2}{ac}
 -\dfrac{3ac}{b^2}-\dfrac{3bd}{c^2}\right) .
\end{align*}
From \eqref{e3.9}, we get
\begin{align*}
&( \dfrac{12bc}{ad}+ 20-\dfrac{c^3}{ad^2} -\dfrac{b^3}{a^2d}
-\dfrac{12c^2}{bd}-\dfrac{12b^2}{ac}
 -\dfrac{3ac}{b^2}-\dfrac{3bd}{c^2})\\
 =&\dfrac{4 (n+1)^2 [- (n+5) k^2 + (n^2+4n-5)k- n^2+1 ]}{k^2 (k+1) (k+2) (n-k+1) (n-k) (n-k-1)^2}.
\end{align*}
Denote
$$
f_4(k):=- (n+5) k^2 + (n^2+4n-5)k- n^2+1.
$$
It is easy to get that function $f_4(k)$ is increasing on
$\Big[1,\dfrac{n-1}{2}\Big]$ and decreasing on
$\Big[\dfrac{n-1}{2},n-2\Big]$, and
$$
f_4(1)=f_4(n-2)=3(n-3)>0,
$$
so we have $f_4(k)>0$ for $1\leq k\leq n-2$. Combining formulas
above we obtain
$$A_1A_2-A_3^2/36>0.$$
\end{proof}

\par
{\bf $\bullet$ Proof of the Claim \eqref{e3.8}.}
\par
Lemma \ref{yl1} implies $\theta_1>0, \theta_2>0$.
 By Lemma
\ref{yl2},
\begin{align*}
A_1\al^2 z_1^2+A_2z_2^2z_3^2\geq&
2\sqrt{A_1A_2}|\al\sigma_3|>\dfrac{1}{3}|A_3\al\sigma_3|,\\
A_1\al^2z_2^2+A_2z_1^2z_3^2\geq&
2\sqrt{A_1A_2}|\al\sigma_3|>\dfrac{1}{3}|A_3\al\sigma_3|,\\
A_1\al^2z_3^2+A_2z_1^2z_2^2\geq&
2\sqrt{A_1A_2}|\al\sigma_3|>\dfrac{1}{3}|A_3\al\sigma_3|.
\end{align*}
Combing the inequalities above, we get
\begin{align*}
V(z)=&A_1\al^2\dsum_{i=1}^3z_i^2+A_2\dsum_{p<q}^3z_p^2z_q^2+A_3\al\sigma_3\geq
0.
\end{align*}



\end{document}